\documentclass[11pt]{article}
\usepackage[overload]{empheq}
\usepackage{latexsym,dsfont,textcomp,amsmath}
\usepackage[english]{babel}
\usepackage[utf8]{inputenc}
\usepackage{enumerate,indentfirst}
\usepackage[colorlinks,citecolor=blue,urlcolor=blue]{hyperref}
\usepackage{amssymb}
\usepackage{amsthm}
\usepackage{hyperref}
\usepackage{graphicx}
\usepackage[active]{srcltx}

\newtheorem{Theorem}{Theorem}[part]
\newtheorem{Definition}{Definition}[part]
\newtheorem{Proposition}{Proposition}[part]

\newtheorem{Lemma}{Lemma}[part]
\newtheorem{Corollary}{Corollary}[part]
\newtheorem{Remark}{Remark}[part]

\newtheorem*{Cita}{Corollary \ref{ThmLaplPont}}

\newcommand{\nc}{\newcommand}
\nc{\esssup}{\mathop{\mathrm{ess\,sup}}}
\nc{\essinf}{\mathop{\mathrm{ess\,inf}}}
\nc{\argmax}{\mathop{\mathrm{arg\,max}}}

\def \P{\mathbb{P}}
\def \N{\mathbb{N}}

\def \R{\mathbb{R}}

\def \E{\mathbb{E}}

\def \D{\mathbb{D}}

\def \1{\mathds{1}}

\def \Ec{{\cal E}}
\def \Fc{{\cal F}}

\def \Qc{{\cal Q}}

\def \Nc{{\cal N}}
\def \Sc{{\cal S}}

\def \l({{\left (}}
\def \r){{\right )}}
\def \l[{{\left [}}
\def \r]({{\right ]}}

\newcommand{\MBFigure}[6]{
$\left. \right.$ \\
\refstepcounter{figure}
\addcontentsline{lof}{figure}{\numberline{\thefigure}{\ignorespaces #5}}
\begin{center}
\begin{minipage}{#1cm}
\centerline{\includegraphics[width=#2cm,angle=#3]{#4}}
\begin{center}
\upshape{F\textsc{ig} \normal
\end{center}
size{\thefigure}. $-$} #5
\end{center}
\label{#6}
\end{minipage}
\end{center}
$\left. \right.$ \\}

\begin{document}

\title{Bessel bridges decomposition with varying dimension. Applications to finance.}
%\author{Gabriel FARAUD \thanks{WIAS Berlin, Mail: Faraud@wias-berlin.de.} {\sc and }St\'ephane GOUTTE \thanks{Laboratoire de Probabilit\'es et Mod\`eles Al\'eatoires,
%CNRS, UMR 7599, Universit{\'e}s Paris 7 Diderot, supported by the FUI project $R=MC^2$, Mail: goutte@math.univ-paris-diderot.fr.}
%}
\author{
    Gabriel FARAUD \\
    Weierstrass Institute for Applied Analysis and Stochastics\\
    WIAS - Berlin\\
    faraud@wias-berlin.de
  \and
    St\'ephane GOUTTE\\
    Centre National de la Recherche Scientifique (CNRS)\\
    Universit{\'e}s Paris 7 Diderot\\
    Laboratoire de Probabilit\'es et Mod\`eles Al\'eatoires\\
    goutte@math.univ-paris-diderot.fr
}

\maketitle
\begin{abstract}
We consider a class of stochastic processes containing the classical and well-studied class of Squared Bessel processes. Our model, however, allows the dimension be a function of the time. We first give some classical results in a larger context where a time-varying drift term can be added. Then in the non-drifted case we extend many results already proven in the case of classical Bessel processes to our context. Our deepest result is a decomposition of the Bridge process associated to this generalized squared Bessel process, much similar to the much celebrated result of J. Pitman and M. Yor. On a more practical point of view, we give a methodology to compute the Laplace transform of additive functionals of our process and the associated bridge. This permits in particular to get directly access to the joint distribution of the value at $t$ of the process and its integral. We finally give some financial applications of our results.
\end{abstract}

\vspace{1cm}


%\textbf{JEL Classification:}\\

\textbf{MSC Classification (2010):60G07 60H35 91B70} 

%\newpage

%\tableofcontents

\section*{Introduction}
%%%%%%%%%%%%%%%%%
%%%%%%%%%%%%%%%%%
We will consider in this article a family of stochastic processes which contains the classical squared Bessel processes. We recall that for $\delta\geq 0$, a $\delta$-dimensional squared Bessel process started at $x\geq 0$ is defined as the unique strong solution of the stochastic differential equation
 \begin{equation*}
 dX_u=\delta du+2\sqrt{X_u}dW_u\quad {and} \quad X_0=x\geq0
 \end{equation*}
for all $u\in \R^+$ where $(W_u)_{u \geq 0}$ denotes a Brownian motion. There has been a huge number of publications concerning this class of processes, resulting in many tractable tools for its study. We refer in particular to Shiga-Watanabe \cite{SW}, Pitman-Yor \cite{PY}, Revuz-Yor \cite{RY} and Deelstra-Debaen \cite{DD1}, \cite{DD2}. A natural extension of this family is to replace the dimension $\delta$ by a function of the time variable. This was done by Carmona in \cite{Carmona}, extending some of the previously mentioned tools to the class of squared Bessel processes with dimension $\delta$ which is a time varying function. A second generalization of this class of processes was done by Deelstra and Delbaen in \cite{GBPhd}, \cite{DD1} and \cite{DD2}. They introduced a constant drift parameter $\beta$ and used a dimension $\delta$ which may be stochastic. Indeed, they extend results of existence of the solution of an extended stochastic differential equation
 \begin{equation*}
 dX_u=(\delta_u+2\beta X_u)du+2\sqrt{X_u}dW_u\quad {and} \quad X_0=x\geq0
 \end{equation*}
They proved that the properties of scaling and additivity of the process $(X_u)$ found by Pitman and Yor in \cite{PY} for classical squared Bessel processes stay valid for these extended squared Bessel processes. Moreover, Pitman and Yor in \cite{PY}, using some classical space-time transformations and Girsanov's theorem, have shown that the study of squared Bessel processes with drift can be reduced to the study of the squared Bessel process without drift (i.e. $\beta=0$). This property is equally true in the case of non-constant dimension. Finally, Delbaen and Shirakawa in \cite{DS} applied the class of squared Bessel processes with time varying dimension $\delta$  to the square root interest rate model.

Hence our contributions are: firstly, we generalize the family of squared Bessel processes to the class of squared Bessel processes with  stochastic dimension $(\delta_u)$ and with a drift parameter $\beta_u$ which is a function of the time. We will call this kind of process a generalized squared Bessel process (GBESQ) with stochastic dimension $(\delta_u)$ and with a drift function $\beta_u$, defined as the unique solution of the stochastic differential equation
 \begin{equation}\label{Int1}
 dX_u=(\delta_u+2\beta_u X_u)du+2\sqrt{X_u}dW_u\quad {and} \quad X_0=x\geq 0.
 \end{equation}
 We establish several classical properties for this generalized process, extending or sometimes overlapping results from the aforementioned articles. This includes results of existence and uniqueness of the solution of \eqref{Int1}; scaling and additivity properties of this solution. We finish by showing that, as in the two other previously mentioned settings, one can reduce to the situation $\beta=0$ through a change of law. Therefore we make a more precise study of the un-drifted case and give some rather tractable tools in order to compute the law of additive functionals of this process. 

We then give a more precise and theoretic version of the last result, in the form of a L\'evy-It\^o representation of this process, similar to the one exposed in Pitman and Yor \cite{PY}. However due to the fact that the dimension parameter belongs to a function space and not, as in the classical setting, to $\R^+,$ we cannot extend to our case all the results of global monotone coupling that exist in the classical setting.  Our tools rely mostly, apart from the probabilistic tool used in the setting of Bessel Processes, on functional analysis in order to extend results on the half line to a functional space. The Fréchet-Riesz representation Theorem is a quite typical example of the theoretical results we use.

Then we introduce the notion of Generalized Squared Bessel Bridges by conditioning our process to its endpoints. We give some extensions of the classical properties of Bessel bridges to our case, including an extension of the Bessel Bridge decomposition of Pitman and Yor \cite{PY}. We use this decomposition in order to give a general formula for the Laplace transform of additive functionals of the Bessel bridge. We give an example of how we can use this formula to get an explicit result. The techniques we use are very tractable and can be applied whenever one knows how to handle classical Bessel processes. We give a (if not very tractable) complete formula covering most settings in Appendix \ref{LTAF}. Our formula contains a power series, the value of which we are unable to compute explicitly, we however give some estimates on the speed of convergence of this series.
 
We conclude this paper by giving some applications to financial mathematics of our results. Indeed, we will first give some examples of financial models which are GBESQ. Then, we will give three examples of applications: first we will give an explicit way to simulate stochastic volatility model where the volatility process is again a GBESQ, then the evaluation of zero coupon bond where the interest rate is a GBESQ, and finally, we will give an explicit method to simulate default times in credit risk model using a stochastic default intensity which is a GBESQ process.

%%%%%%%%%%%%%%%%%
%%%%%%%%%%%%%%%%%
\section{Squared Bessel processes with time varying dimension.}

\subsection{Squared Bessel process} 

Let $\left(\Omega,(\Fc_t)_{t\in [0,T]},\P\right)$ be a filtered probability space and let $(W_t)_{t\in [0,T]}$ be a Brownian motion under $\P$. We suppose that the filtration $(\Fc_t)_{t\in [0,T]}$ is generated by $\left\{W_s;0\leq s \leq t\right\}$.  We first recall some results of \cite{RY}, about squared Bessel processes with constant dimension $\delta\geq0$.

 \begin{Definition}\label{DefBESQ}
 The squared Bessel process of dimension $\delta\geq0$, denoted by BES$Q^{\delta}_x$ is defined as the unique strong solution of the stochastic differential equation (SDE)
 \begin{equation}\label{EqBESQ}
 dX_u=\delta du+2\sqrt{X_u}dW_u\quad {and} \quad X_0=x\geq 0.
 \end{equation}
 In this case $X_u \geq 0$, for all $u \in [0,T].$
 \end{Definition}
 
We will denote in the sequel by $\Qc_x^\delta$, the law of BES$Q^{\delta}_x$. %on $C(\R^+,\R^+)$ 
We now recall the property of additivity of squared Bessel processes which can be found in Theorem 1.2 Chap XI of \cite{RY}
\begin{Theorem}\label{ThmRY}
For every $\delta,\delta^{'}\geq 0$ and $x,x^{'}\geq 0$, we have
\begin{equation}\label{eqRY}
\Qc_x^\delta\oplus \Qc_{x^{'}}^{\delta^{'}}=\Qc_{x+x^{'}}^{\delta+\delta^{'}}.
\end{equation}
\end{Theorem}

\subsection{Generalized squared Bessel process (GBESQ)}

Let $C$ be the space of continuous functions $\R^+\rightarrow \R^+$ and $D$ the space of continuous by part, measurable and locally bounded functions $\R^+\rightarrow \R^+$. 

\begin{Definition}\label{Def1}
For all $u\in \R^+$ , $\delta,\beta \in D$, we will call the solution in $C$ of the stochastic differential equation
\begin{equation}\label{EqDef1}
dX_u=\left(2\beta_uX_u+\delta_u\right)du+2\sqrt{X_u}dW_u
\end{equation}
with initial value $X_0=x\geq 0$, a generalized squared Bessel process with dimension $(\delta_u)$ and drift function $(\beta_u)$. We will denote $X$ as a $GBESQ^{\beta,\delta}_x.$
\end{Definition}

We will denote by $^{\beta}\Qc^{\delta}_x$ the law of X and
$\Qc^{\delta}_x:=^{0}\Qc^{\delta}_x$.

\begin{Remark}
In the above definition $(\delta_u)$ and $(\beta_u)$ are deterministic functions of the time. It is obviously possible to extend these definitions to random $(\delta_u)$ and $(\beta_u)$ provided their trajectories belong almost surely to $D$, by conditioning on the realization of $(\delta_u)$ and $(\beta_u)$. Thus all the following result are true conditionally on $(\delta_u)$ and $(\beta_u)$, or, using the typology of random media, in the "quenched" setting. It is also straightforward by time restriction of the above SDE that the solution in this case is progressively measurable with respect to $\sigma((\delta_u)_{u\le t},(\beta_u)_{u\le t},(W_u)_{u\le t}).$
\end{Remark}

\subsubsection{Existence, unicity and positivity of the solution}

We will obtain the uniqueness of the existence of the GBESQ and its positivity for a more general volatility structure. Indeed, we take the process $X$ given by
\begin{equation}\label{EqDef2}
dX_u=\left(2\beta_uX_u+\delta_u\right)du+\sigma(X_u)dW_u
\end{equation}
where $\sigma:\R^+\rightarrow\R^+$ is a function, vanishing at zero and satisfies the Holder condition, namely that for all $x,y \in \R^+$, there exists a constant $c$ such that
\begin{equation}\label{ConditionVol}
\vert \sigma(x)-\sigma(y)\vert \leq c \sqrt{\vert x-y\vert}.
\end{equation}
\begin{Remark}\label{RemarkDeelstra}
\begin{enumerate}
\item We can remark that the volatility structure of our GBESQ process \eqref{EqDef1} satisfies this condition.
\item This class of processes was studied by Deelstra in \cite{GBPhd} with a constant drift $\beta$ and a dimension $\delta$ such that $\int_0^t \delta_u du<\infty$ for all $t\in \R^+$.
\item If $\delta$ is in $D$ then $\delta$ satisfies point 2. 
\end{enumerate}
\end{Remark}

\begin{Proposition}
Let $\beta$ and $\delta$ be in $D$ and $\sigma$ a function satisfying \eqref{ConditionVol}, then for every $x\geq 0$ there is a unique solution $X_t$ of the stochastic differential equation \eqref{EqDef2} such that $X_0=x$.
\end{Proposition}
\begin{proof}
The proof follows from Theorem 3.2.1 in \cite{GBPhd}. Whereas as we said in Remark \ref{RemarkDeelstra} point 2. our model has a non constant drift term. But since $\beta$ is in $D$, then it is locally bounded. Hence we can work with this new drift in the proof of Theorem 3.2.1 and obtain the same result.
\end{proof}
To obtain the positivity of this solution, we will use a comparison result.
\begin{Proposition}\label{ThmCompa}
Let $X^1$ and $X^2$ be two solution of (\ref{EqDef2}) with dimension $\delta^1\in D$ and $\delta^2 \in D$ and respectively initial state $x^1>0$ and $x^2>0$. We suppose too, that the processes $X^1$ and $X^2$ are associated with the same Brownian motion $(W_u)$ and the same drift  $\beta \in D$. If for all $u \geq 0$ we have $x^2\geq x^1$ and $\delta^2_u\geq \delta^1_u$ a.s., then for every $u \geq 0$, we have
$$
P\left[X^2_u\geq X^1_u\right]=1.
$$
\end{Proposition}
\begin{proof}
We will again follow the proof of Theorem 3.2.2 in \cite{GBPhd} with a non constant drift $\beta$ using the fact that $\beta_u$ is in $D$ for all $u\in [0,T]$.
\end{proof}

\begin{Corollary}\label{CoroPositive}
The process $X$ is always positive. 
\end{Corollary}
\begin{proof}
This follows directly from the Proposition \ref{ThmCompa}. Indeed, taking $X^1_0=0$ and $\delta^1\equiv 0$, then the unique solution of (\ref{EqDef2}) is given by $X^1\equiv 0$. Then taking $X^2$ another solution of \eqref{EqDef2} such that  $X^2_0\geq X^1_0=0$ and $\delta^2_u\geq \delta^1_u\equiv 0$ a.s. we obtain that $X^2_u\geq 0$  a.s.
\end{proof}

\begin{Remark}
Hence taking for all $u \in [0,T]$,  $\sigma(X_u)=2\sqrt{X_u}$, then we have that the stochastic differential equation \eqref{EqDef1} admits a unique positive solution.

\end{Remark}

\subsection{Results about GBESQ}

We come back now to our GBESQ process given as the unique solution of the stochastic differential equation \eqref{EqDef1}

\begin{Proposition}\label{PropoAdditivite} 
Let $(X^{(\delta)}_u)_{u \in [0,T]}$ and $(Y^{(\delta^{'})}_u)_{u \in [0,T]}$ be two independent generalized squared Bessel processes with time varying dimension $\delta_u,\delta_u^{'}\geq 0$ with the same drift $\beta_u$ which is a bounded function and initial states $x,x^{'}\geq 0$. Then
\begin{equation}\label{eqAdd}
\{X^{(\delta)}_u+Y^{(\delta^{'})}_u; 0\leq u, X^{(\delta)}_0=x,Y^{(\delta^{'})}_0=x^{'}\}\stackrel{law}{=}\{H^{(\delta+\delta^{'})}_u; 0\leq u,H^{(\delta+\delta^{'})}_0=x+x^{'}\},
\end{equation}
where $H$ is another GBESQ with stochastic dimension $\delta_u+\delta_u^{'}$, with drift equal to $\beta_u$ and initial state $x+x^{'}$. In other words,
\begin{equation}
^\beta \Qc^{\delta}_x \oplus ^\beta \Qc^{\delta'}_{x'}=^\beta \Qc^{\delta+\delta'}_{x+x'}.
\end{equation}
\end{Proposition}
\begin{proof}
We proceed similarly to the proof of Theorem 1.2 of \cite{RY} which is done for simple squared Bessel processes. Let $W$ and $W^{'}$ be two independent linear Brownian motion. 
Let $X$ and $Y$ defined by \eqref{EqDef1} as the corresponding two solutions for $(x,\beta,\delta)$ and $(x^{'},\beta,\delta^{'})$ and set H=X+Y. Then for any $u \in [0,T]$,
\begin{eqnarray*}
H_u=x+x^{'}+\int_0^u2\beta_s\left(X_s+Y_s\right)ds+\int_0^u\left(\delta_s+\delta^{'}_s\right)ds+2\int_0^u\left(\sqrt{X_s}dW_s+\sqrt{Y_s}dW^{'}_s\right).
\end{eqnarray*}
Let $W^{''}$ be a third Brownian motion independent of $W$ and $W^{'}$. The process $\zeta$ defined by
\begin{eqnarray*}
\zeta_u=\int_0^u1_{\{H_s>0\}}\frac{\sqrt{X_s}dW_s+\sqrt{Y_s}dW^{'}_s}{\sqrt{H_s}}+\int_0^u1_{\{H_s=0\}}dW^{''}_s
\end{eqnarray*}
is a linear Brownian motion since $\left \langle \zeta,\zeta \right \rangle_u=u$ and we have
\begin{eqnarray*}
H_u&=&(x+x^{'})+\int_0^u2\beta_s\left(X_s+Y_s\right)ds+\int_0^u\left(\delta_s+\delta^{'}_s\right)ds+2\int_0^u\sqrt{H_s}d\zeta_s\\
H_u&=&(x+x^{'})+\int_0^u\left(2\beta_sH_s+\left(\delta_s+\delta^{'}_s\right)\right)ds+2\int_0^u\sqrt{H_s}d\zeta_s,
\end{eqnarray*}
which completes the proof.
\end{proof}

\begin{Proposition}\label{PropoScaling1}
For all $u\in [0,T]$, If X is a $GBESQ^{\beta,\delta}_x$ (i.e. $X_0=x$), then for any $c>0$, the process $\frac{1}{c}X_{cu}$ is again a $GBESQ^{c\beta^{'},\delta^{'}}_{\frac{x}{c}}$ with $\delta^{'}_u=\delta_{cu}$ and $\beta^{'}_u=\beta_{cu}$. In terms of distribution, this reformulates to
$$^\beta \Qc^\delta_x\left(\frac{1}{c}X_{c\, \cdot}\right)=^{c\beta(c\, \cdot)} \Qc^{\delta(c\, \cdot)}_{x/c}.$$
\end{Proposition}
\begin{proof}
We know that $X_u=x+\int_0^u\delta_sds+\int_0^u2\beta_sX_sds+2\int_0^u\sqrt{X_s}dW_s$. So by a change of variable in this stochastic integral, we obtain
\begin{eqnarray*}
X_{u}&=&x+\int_{0}^{\frac{u}{c}}c\delta_{cv}dv+\int_0^{\frac{u}{c}}2c\beta_{cv}X_{cv}dv+2\int_{0}^{\frac{u}{c}}\sqrt{X_{cv}}dW_{cv}.
\end{eqnarray*}
Hence
\begin{eqnarray*}
\frac{1}{c}X_u&=&\frac{x}{c}+\int_{0}^{\frac{u}{c}}\delta_{cv}dv+\int_0^{\frac{u}{c}}2\beta_{cv}X_{cv}dv+2\int_{0}^{\frac{u}{c}}\frac{1}{c}\sqrt{X_{cv}}dW_{cv},\\
\frac{1}{c}X_u&=&\frac{x}{c}+\int_{0}^{\frac{u}{c}}\delta_{cv}dv+\int_0^{\frac{u}{c}}2\beta_{cv}X_{cv}dv+2\int_{0}^{\frac{u}{c}}\sqrt{\frac{1}{c}X_{cv}}\frac{1}{\sqrt{c}}dW_{cv},\\
\frac{1}{c}X_{cu}&=&\frac{x}{c}+\int_{0}^{u}\delta_{cv}dv+\int_0^u2\beta_{cv}X_{cv}dv+2\int_{0}^{u}\sqrt{\frac{1}{c}X_{cv}}\frac{1}{\sqrt{c}}dW_{cv}.
\end{eqnarray*}
Moreover, we know that if for all $u \in [0,T]$, $W_u$ is a Brownian motion then, for all $c>0$, $\tilde{W}_{u}:=\frac{1}{\sqrt{c}}W_{cu}$ is again a Brownian motion. Hence this stochastic differential equation becomes
\begin{eqnarray*}
\frac{1}{c}dX_{cu}&=&\left(2\beta_{cu}X_{cu}+\delta_{cu}\right)du+2\sqrt{\frac{1}{c}X_{cu}}d\tilde{W}_{u}=\left(2c\beta_{cu}\frac{1}{c}X_{cu}+\delta_{cu}\right)du+2\sqrt{\frac{1}{c}X_{cu}}d\tilde{W}_{u}.
\end{eqnarray*}
Consequently, the result follows from the uniqueness of the solution to this stochastic differential equation.
\end{proof}

\subsubsection{Removal of the drift term.}
We now give a generalization of the change of law property which can be found in \cite{PY}. This proposition allows us to vanish the deterministic drift function. Let $X$ be a $GBESQ^{0,\delta}_x$ and for $t \in \R^+$ define the process $Y$ by $Y_t=\int_0^t \beta_s\sqrt{X_s}dW_s$, where the function $\beta$ is in $D$ and assume moreover that $\beta$ is differentiable. Then the continuous local martingale defined by
$$
Z_t=\Ec(Y_t)=\exp\left\{Y_t-\frac{1}{2}\langle Y,Y\rangle_t \right\}
$$
is equal to
\begin{eqnarray*}
Z_t&=&\exp\left\{\int_0^t \beta_s\sqrt{X_s}dW_s-\frac{1}{2}\int_0^t \beta^2_sX_sds \right\}\\
%&=&\exp\left\{\int_0^t \beta_s\sqrt{X_s}dW_s+\int_0^t2\beta_s\delta_sds-\int_0^t2\beta_s\delta_sds-\frac{1}{2}\int_0^t \beta^2_sX_sds \right\}\\
&=&\exp\left\{\frac{1}{2}\left[\int_0^t 2\beta_s\sqrt{X_s}dW_s+\int_0^t\beta_s\delta_sds-\int_0^t\beta_s\delta_sds-\int_0^t \beta^2_sX_sds \right]\right\}\\
&=&\exp\left\{\frac{1}{2}\left[\int_0^t \beta_s\left(\delta_sds+2\sqrt{X_s}dW_s\right)-\int_0^t\beta_s\delta_sds-\int_0^t \beta^2_sX_sds \right]\right\}.
\end{eqnarray*}
Since $X$ is a  $GBESQ^{0,\delta}_x$, we obtain that
\begin{eqnarray}\label{P1}
Z_t&=&\exp\left\{\frac{1}{2}\left[\int_0^t \beta_sdX_s-\int_0^t\beta_s\delta_sds-\int_0^t \beta^2_sX_sds \right]\right\}.
%&=&\exp\left\{\frac{1}{2}\left[\int_0^t \beta_sdM_s-\int_0^t \beta^2_sX_sds \right]\right\}\nonumber
\end{eqnarray}
%with $M_s=X_s-\delta_ss$. 
By integrating by parts, we have
$$
\beta_tX_t=\beta_0X_0+\int_0^t\beta_sdX_s+\int_0^tX_sd\beta_s=\beta_0X_0+\int_0^t\beta_sdX_s+\int_0^tX_s\beta^{'}_sds
.$$
Hence by substituting this expression in \eqref{P1} we obtain
\begin{eqnarray*}
Z_t&=&\exp\left\{\frac{1}{2}\left[\beta_tX_t-\beta_0X_0-\int_0^tX_s\beta^{'}_sds-\int_0^t\beta_s\delta_sds-\int_0^t \beta^2_sX_sds \right]\right\}\\
&=&\exp\left\{\frac{1}{2}\left[\beta_tX_t-\beta_0X_0-\int_0^t\beta_s\delta_sds-\int_0^t \left(\beta^{'}_s+\beta^2_s\right)X_sds \right]\right\}.
\end{eqnarray*}
If we show that $Z_t$ is a true-martingale then we will can apply the change of law formula given by the following proposition
\begin{Proposition}\label{Girsanov}
\begin{equation}\label{DefZGirsanov}
Z_t:=\frac{^\beta \Qc^\delta_x}{\Qc^\delta_x}=\exp\left\{\frac{1}{2}\left[\beta_tX_t-\beta_0X_0-\int_0^t\beta_s\delta_sds-\int_0^t \left(\beta^{'}_s+\beta^2_s\right)X_sds \right]\right\}.
\end{equation}
\end{Proposition}
In fact, since $X^{0,\delta}_s$ is positive for $s \in [0,T]$ and $\beta\in D$, so $\beta$ is locally bounded. We have that $\left(\beta^{'}_s+\beta^2_s\right)\geq 0$ then the local-martingale $Z_t$ is bounded, therefore it is true-martingale.

\subsection{Some law results}

Starting from now and having in mind the previous result, we will always assume that the drift term is zero. We first recall a result from
\cite{Carmona}, Proposition 3.4
\begin{Proposition}\label{PropoCarmona2}
If $\delta\in D$
 and $\mu$ is a positive Radon measure on $[0,a[$,
then
\begin{equation*}
\Qc^{\delta}_{x}\left[\exp\left(-\int_0^a X_ud\mu(u)\right)\right]=\exp\left(\frac{x}{2}\Phi^{'}_\mu(0)\right)\exp\left(\frac{1}{2}\int_0^a\frac{\Phi^{'}_\mu(u)}{\Phi_\mu(u)}\delta_u du\right),
\end{equation*}
where $\Phi_\mu$ is the (well and uniquely defined) solution in the distribution sense of \begin{equation}\label{EqPhiS}
\Phi^{''}=\Phi \mu.
\end{equation} which is positive, non-increasing on $[0,\infty[$ and follows the initial condition $\Phi_\mu(0)=1$.\\
\end{Proposition}

\begin{Remark}\label{LaplGEBSQ}
\begin{enumerate}
\item
With straightforward calculation, we can obtain a extended version of this result with non zero drift term, namely 
\begin{equation*}
{}^\beta \Qc^{\delta}_{x}\left[\exp\left(-\int_0^a X_ud\mu(u)\right)\right]=\exp\left(\frac{x}{2}\Phi^{'}_\mu(0)\right)\exp\left(\frac{1}{2}\int_0^a\frac{\Phi^{'}_\mu(u)}{\Phi_\mu(u)}\delta_u du\right),
\end{equation*}
where $\Phi_\mu$ the unique solution of the differential equation, in the distribution sense,
\begin{equation}\label{EqSDEPhi}
d\mu_u=\frac{\Phi^{''}_u}{2\Phi_u}du+\beta_u\frac{\Phi^{'}_u}{\Phi_u}du
\end{equation}
which is positive, non-increasing on $[0,\infty[$ and follows the initial condition $\Phi_\mu(0)=1$.\\
\item All these results could be extended under proper assumptions to measures without compact support. As our goal is to study the Bessel bridge, we feel free to do this restriction.
\end{enumerate}
\end{Remark}

The previous form gives a precise explicit expression for the Laplace transform of additive functional of the generalized squared Bessel. However in practical cases this explicit form is rather difficult to use. We give a more abstract form which, surprisingly, suits better our goals, and allows us to introduce some notations. 
\begin{Proposition}
\label{riesz}
Let $X^\delta$ be a generalized squared Bessel process with dimension $\delta$ and started in $x$.  If $\delta \in D$ and $\mu$ is a positive Radon measure with compact support on $\R^+$, then there exist a constant $A_{\mu}$ and a function $B_{\mu}:\R^+\rightarrow \R^+$ such that 
\begin{equation*}
\Qc^{\delta}_{x}\left[\exp\left(-\int_0^\infty X_ud\mu(u)\right)\right]=A_{\mu}^x \exp\left\{\int_{0}^{\infty}B_{\mu(u)}\delta_{u}du\right\}.
\end{equation*}
\end{Proposition}
\begin{proof}
It is a trivial corollary of Proposition \ref{PropoCarmona2}.

\end{proof}
\begin{Remark}
Note that, by taking $\delta$ constant, the $A_{\mu}$ term is the same as in Corollary 1.3 of page 440 of \cite{RY}.
\end{Remark}

We can apply this formula to measures of the form $\mu_{t}:=\lambda \epsilon_{t}$, where $\epsilon_{t}$ stands for the Dirac measure in $t$. Our strategy to compute $A_{\mu_{t}}$ and $B_{\mu_{t}}$ will be to take well-chosen test values and functions for $x$ and $\delta$.
Taking $\delta=1$,
$A_{\mu_{t}}$ was computed in \cite{RY}, page 441, and is equal to
\begin{equation}\label{EqA}
A_{\mu_{t}}=e^{-\frac{\lambda}{1+2\lambda t}}.
\end{equation}
To get access to $B_{\mu_{t}}$, we take $x=0$ and use test functions of the form $\delta=\chi_{[b,\infty]}$ where $\chi$ denotes the indicator function. For such $\delta$ and $x=0$, the generalized squared Bessel process is easily described: indeed it is identically equal to $0$ for $0\le t\le b$ and then evolutes like a squared Bessel process with dimension $1$ started at $0$, that is, the square of a standard Brownian motion.
Therefore straightforward computations imply
\begin{equation}\Qc^{\chi_{[b,\infty]}}_{0}\left[\exp\left\{-\int X_{u}d\mu_{t}\right\}\right]=E_{0}\left[e^{-\lambda B_{(t-b)^+}^2}\right]
=\frac{1}{\sqrt{1+2\lambda(t-b)^+}}.
\end{equation}
(for the last equality see for example p441 of \cite{RY}.)
On the other hand this is also equal to $\exp\left\{\int_{b}^{\infty}B\mu(u)\right\}$. Taking the log and then the derivative with respect to $b$, then bringing everything together, we get the following result :
\begin{Corollary}\label{Laplace} Under the statement of the Proposition \ref{riesz}, we have
$$\Qc^{\delta}_{x}\left[e^{-\lambda X_{t}}\right]=\exp\left\{-\frac{\lambda x}{1+2\lambda t}\right\}\exp\left\{-\int_{0}^t \frac{\lambda \delta_{u}}{1+2\lambda (t-u)}du\right\}.$$

\end{Corollary}
This is the Laplace transform of the transition Kernel of $\Qc^{\delta}_{x}$. In the case of constant $\delta$ an inverse to this transform can be computed and therefore an expression of the transition Kernel, it is however not generally possible for any $\delta$. We give in Remark \ref{RemarkLaplaceSimplifier} a simplification of this formula in the case of piecewise constant $\delta$, and in section \ref{SecGenerating} some hints on how to overcome this issue in numerical simulations.

\section{The Levy-It\^o Representation of GBESQ}

We now turn to the so-called Levy-It\^o representation. For a motivation of such representation we refer to \S 4 of \cite{PY}.

\begin{Theorem}\label{THM23}
\begin{enumerate}
\item There exists a unique measure $M$ on $C$ such that for every random variable $I$ on $C$ of the form $I=\int X_{t} \mu(dt),$ with $\mu$ a positive Radon measure with compact support on $(0,\infty)$ and every $\alpha>0,$ denoting by $M_{u}$ the distribution of $(X_{(t-u)^+}, t\ge 0)$ for $(X_{t}, t\ge 0)$ distributed according to $M$ we get
\begin{equation}
\Qc^\delta_{x}\left[e^{-\alpha I}\right]=\exp\left\{\left(xM+\int_{0}^\infty \delta_{u}M_{u} du\right)\left(e^{-\alpha I} -1\right)\right\}.\label{LK}
\end{equation}
Moreover $M$ corresponds to the $^0 M$ of Theorem 4.1 of \cite{PY}. 
\item For any $\delta\in D,$ there exists a $R^+\rightarrow C$ process $(Y_{x}^\delta, x\ge 0)$ such that
\begin{equation}\label{somme}
Y_{x}^{\delta}=Y_{x}^0+Y_{0}^{\delta}=: Y_{x}+ Y^{\delta},
\end{equation}
where $(Y_{x}, x\ge 0)$ and $(Y^\delta, \delta \in D)$ are independent, $Y_{x}$ is as in Theorem 4.1. of \cite{PY} and $Y^\delta$ is distributed according to $\Qc_0^{\delta}.$

\end{enumerate}
\end{Theorem} 
\begin{Remark}\label{remark}

\begin{enumerate}\item
This theorem should be understood as a L\'evy-Khinchine formula on $C$. For more on this we refer to \S4  of \cite{PY}.
We also refer to \S3 of \cite{PY} for background on the measure $M$. In short, $M$ is the excursion law associated to the zero diffusion $\Qc_x^0$ renormalized so that its entrance law is given by 
$$M(X_t\in dx)=\frac{dx}{(2t)^2}e^{-\frac{x}{2t}}.$$
In particular the associated transition kernel starting for any strictly positive time is the same as the one of the squared Bessel process of dimension $0$.

\item We also recall the expression of $Y_x$ from \cite{PY}. Namely, 
$$Y_x(0)=x,\; Y_x=\sum_{v\le x} \Delta_{v} \text{ on }(0,\infty),$$
where $\Delta_v$ is a Poisson point process on $C$ with Poisson measure $M$.
Note that, due to the form of $M$, the sum in the last expression is actually infinite, however, for every finite $t$, the set of $\Delta_v$ such that $\Delta_v(t)>0$ is finite and its cardinality follows a Poisson distribution of parameter $\frac{x}{2t}.$ 
\item In \cite{PY}, the representation in terms of a Poisson point process is also valid for $\delta$. However we do not have such a representation in our setting and believe it is a difficult question. It is however not needed for our purpose.
\end{enumerate}
\end{Remark}
\begin{proof}[Proof of Theorem \ref{THM23}:]
First, we would like to prove \eqref{LK}.
%Note that $M$ is obviously the same as the $^0 M$ of Theorem 4.1. of \cite{PY}. Therefore the ``$x$'' part of the decomposition is straightforward. We therefore concentrate on the case $x=0$. 
It is clear that both terms are multiplicative in $\delta$ and positive. The form in Proposition \ref{PropoCarmona2} also easily imply that the left hand side is continuous in $\delta$ with respect to the infinite norm. The continuity of the right hand side follows easily once  one knows that $\int_{0}^\infty \delta_{u}M_{u} du\left(e^{-\alpha I} -1\right)$ is finite, which is easily deduced from Theorem 4.1. of \cite{PY}.  Hence we just need to prove it for well chosen test function. We take for test functions the functions that are constant on an interval and equal to zero outside. It is easy to see that such function form a generating family of $D$. Indeed it is rather easy to see that any constant by parts can be constructed as a sum of such functions.
The fact that constant by part functions are dense in $D$ is classical.

More precisely, let $\delta:=d \chi_{[a,b]}$. We can easily see that for such $\delta$, $X_{t}$ is identically zero on $[0,a]$ then evolves like a $d$-dimensional squared Bessel process on $[a,b]$, then like a $0$-dimensional squared Bessel process. It is the clear that by a shifting of time we can reduce to the case $a=0$. The Markov property implies:
\begin{equation*}
\Qc_{x}^{\delta}\left[e^{-\alpha I}\right]=\Qc_{x}^d \left[\exp\left\{-\alpha\int_{0}^b X_{t}d\mu(t)\right\} \E_{X_{b}}^0\left[\exp\left\{-\alpha\int_{0}^\infty X_{t}d\mu(t+b)\right\}\right]\right].
\end{equation*}
We write $I_{1}:=\int_{0}^b  X_{t}d\mu(t)$ and $I_{2}=\int_{0}^\infty X_{t}d\mu(t+b)$. Then the previous expression can be written 
\begin{equation}
\Qc_{x}^{\delta}\left[e^{-\alpha I}\right]=\Qc_{x}^d \left[e^{-\alpha I_{1}} \Qc_{X_{b}}^0 \left[e^{-\alpha I_{2}}\right]\right].
\end{equation} 
Theorem 4.1 of \cite{PY} implies that the previous expression is equal to
$$\Qc_{x}^d \left[e^{-\alpha (I_{1}+X_bM(e^{-\alpha I_2}-1))}\right].$$
Denoting by $I'_1$ the function on $C$ defined by 
$$
I'_1((X_t)_{t\ge 0})=I_1((X_t)_{t\ge 0})+X_b M(e^{-\alpha I_2}-1)
$$ 
and applying once again Theorem 4.1. of \cite{PY}, we get
$$\Qc^{\delta}_x\left[e^{-\alpha I}\right]=\exp\left\{\left( xM+ d\int_0^\infty M_u du \right)(e^{-\alpha I'_1}-1)\right\}.$$
It is clear from the previous form of $I'_1$ that for every $u>b,$ $M_u(e^{-\alpha I'_1}-1)=0$, therefore the last expression can be reduced to
$$\Qc^{\delta}_x\left[e^{-\alpha I}\right]=\exp\left\{\left( xM+ d\int_0^b M_u du \right)(e^{-\alpha I'_1}-1)\right\}.$$
We now want to prove that the above expression equals 
$$
\exp\left\{\left( xM+ d\int_0^b M_u du \right)(e^{-\alpha I}-1)\right\}.
$$ 
For this we only need to prove that for every $0\le u< b,$
$$
M_u(e^{-\alpha I'_1}-1)=M_u(e^{-\alpha I}-1).
$$
We get this basically by doing the previous computations backward. Indeed
\begin{eqnarray*}M_u(e^{-\alpha I'_1}-1))&=&M_u\left[\exp\left\{-\alpha\int_0^b X_t d\mu(t)\right\}\exp\{X_b M(e^{-\alpha I_2}-1\}\}-1\right]\\
&=&M_u\left[\exp\left\{-\alpha\int_0^b X_t d\mu(t)\right\}\Qc_{X_{b}}^0\left[\exp\{-\alpha\int_{0}^\infty X_{t}d\mu(t+b)\}\right]-1\right]\\
&=&M_u\left[\exp\left\{-\alpha\int_0^\infty X_t d\mu(t)\right\}-1\right],
\end{eqnarray*}
where we used in the second line Theorem 4.1 of \cite{PY} and in the third line the Markov property for $M_u$ and the fact that, by definition of $M$, for $u<b,$ the transition kernel of $M_u$ for times greater than $d$ is equal to the transition kernel of $\Qc_x^0$.

Now we need to show \eqref{somme}: it is trivial, just by taking $Y_x$ as in \cite{PY} and any independent $\delta-$squared Bessel process started at zero for $Y^\delta$. Then Proposition \ref{PropoAdditivite} gives the result.
\end{proof}

\section{Bessel bridge Decomposition}

We now turn to our main interest, namely the squared Bessel bridge with non-constant dimension.
We define for $\delta\in D$ the law \begin{equation}\Qc^{\delta,t}_{x\to y} :=\Qc^{\delta}(.|X_0=x,X_t=y).\label{defbridge}\end{equation}
In all the result we are going to mention, we are only interested in the trajectory up to time $t\in [0,T]$ of the Bessel bridge. So with a slight abuse of notation we will consider $\Qc^{\delta,t}_{x\rightarrow y}$ as a probability measure on $C([0,t],R^+).$
We will  also make use of the time reversed Bessel bridge, namely we call $\hat{\mathcal{Q}}$ the $\mathcal{Q}$ distribution of $(X(t-s))_{0\le s\le t}$. In particular it is easily seen that we have
\begin{Lemma} Let $X$ be a $GBESQ^{\delta,x}$ such that for $t\geq 0$, $X_t=y$, then we have that
$$\hat{\mathcal{Q}}^{\delta,t}_{x\to y}=\mathcal{Q}^{\hat{\delta},t}_{y\to x},$$
where $\hat{\delta}(s)=\delta(t-s).$
\end{Lemma}
We also restrict to the case where $\delta$ is not equal to zero excepted on a discrete set of points, as it will be the case in financial applications. This is only to ensure that for every $y\ge 0,$ the conditioning event in \eqref{defbridge} has zero probability. It will also ensure that the distribution $\Qc^{\delta,t}_{x\rightarrow y}$ is continuous with respect to $(x,y)$.

We recall that, by definition, $\mathcal{Q}^{0,t}_{0\to y}$ is the time reversal of $\mathcal{Q}^{0,t}_{y\to 0}.$ We first give a "bridge" version of the additivity property. 

\begin{Proposition}\label{additivite_bridge}
Let $\delta$, $\delta'$ be in $D$, then
$$\mathcal{Q}^{\delta,t}_{x\to 0}\oplus\mathcal{Q}^{\delta',t}_{y\to 0}=\mathcal{Q}^{\delta+\delta',t}_{x+y\to 0}.$$
\end{Proposition}

\begin{Remark}
In this setting there is no straightforward representation in term of the unconditioned process in the sense of (5.a) of \cite{PY}, whereas doubtless a careful application of the arguments of \cite{Watanabe} should lead to some similar (but not equal) result. This is however not needed for the above proposition to be true.
\end{Remark}
\begin{proof}[Proof of Proposition \ref{additivite_bridge}]
This is directly obtained by conditioning the statement of Proposition \ref{PropoAdditivite}, using the (obvious) fact that, for $X$ and $Y$ two $\R^+$-valued processes, the events $\{Y(t)+X(t)=0\}$ and $\{X(t)=0,Y(t)=0\}$ are equal. 
\end{proof}

We are aiming at obtaining an extended Bessel bridge decomposition. Firstly, we need some notation and an introductory Lemma.
We denote by $\mathcal{L}^{-1}f$ the inverse Laplace transform of $f$. And we call 
\begin{equation}\label{EqF}
F^t_{\delta}(\lambda)=\exp\left\{-\frac{\lambda x}{1+2\lambda t}\right\}\exp\left\{-\int_{0}^t \frac{\lambda \delta_{u}}{1+2\lambda (t-u)}du\right\}.
\end{equation}
\begin{Lemma}\label{cond}
Let 
$$b_{\delta,x,y}^t(n):=\frac{\frac{(x/(2t))^n}{n!}\mathcal{L}^{-1}F^t_{\delta+2n}(y)}{\sum_{k=0}^\infty\frac{(x/(2t))^k}{k!}\mathcal{L}^{-1}F^t_{\delta+2k}(y)},$$ then 
$b_{\delta,x,y}^t(n)$ is the distribution of U given $$X_{*}+X_1+\dots+X_U=y,$$ where $U\sim Poisson(x/(2t))$, $X_i\sim exp(1/(2t))$ are i.i.d. random variables and $X_*$ is a random variable independent of $U$ and $X_i$ and its Laplace transform is given by $F^t_{\delta}(\lambda)$.
\end{Lemma}

\begin{proof}
By simple computation one gets that the Laplace transform of $X_1$ is equal to $$\exp\left\{-\int_{0}^t \frac{2\lambda }{1+2\lambda (t-u)}du\right\}=\frac{1}{1+2t\lambda}.$$
We deduce that, conditionally on $U=n,$ the Laplace transform of $X_{*}+X_1+\dots+X_U$ is equal to 
 $F_{\delta+2n}^t(\lambda)$. The continuous version of the Bayes Formula then gives the result. 
\end{proof}

Then we obtain the extended Bessel bridge decomposition
\begin{Theorem}\label{ThmBridge}
For every $\delta\in D,$ $x\in \R^+$, $y\in \R^+$,
$$\mathcal{Q}^{\delta,t}_{x\to y}=\mathcal{Q}^{0,t}_{x\to 0}\oplus \mathcal{Q}^{0,t}_{0\to y}\oplus \mathcal{Q}^{\delta,t}_{0\to 0}\oplus \sum_{n=0}^\infty b_{\delta,x,y}^t(n)\mathcal{Q}_{0\to 0}^{4n,t},$$
where the last term on the right is the mixture of the laws $\mathcal{Q}_{0\to 0}^{4n,t},\,n=\{0, 1, \dots \}$ with weights given by $b_{\delta,x,y}^t(n)$.
\end{Theorem}
\begin{proof}
Using point 2. of Remark \ref{remark},  we get that, for $Y_x^{\delta}\sim\mathcal{Q}_x^{\delta},$
$$Y_x^{\delta}=Y_{x}^0+Y_0^{\delta},$$
where the two processes are independent, $Y_0^{\delta}\sim \mathcal{Q}_0^{\delta}$ and 
$$Y_{x}^0=\sum_{v\le x}\Delta_v,$$ where $\Delta_v$ is a Poisson point process with measure $M$. We separate as in \cite{PY} this sum in two, keeping on one side the terms which are zero at $t$ and on the other side the terms which are positive at $t$, we get that
$$Y_x^{\delta}=Y_{x0}^0+Y_{x+}^0+Y_0^{\delta},$$
where the three processes are independent, $$Y_{x0}^0\sim \mathcal{Q}^{0,t}_{x\to 0},$$
and $$Y_{x+}^0=\sum_{i=0}^U Z_i,$$
where U is a $Poisson(x/(2t))$ variable independent of everything else and $Z_i$ are i.i.d. trajectories whose distribution is given by $M$ conditioned on $X_t>0$ and renormalized. We condition on the values of $Z_i$ and $Y_0^{\delta}$ at time $t$.
Conditionally on $U=u,$ and $Z_i(t)=z_i,\; \forall i\le U,$ it is shown in \cite{PY} that
$$Z_i\sim \mathcal{Q}^{4,t}_{0\to z_i} =\mathcal{Q}^{0,t}_{0\to z_i}\oplus\mathcal{Q}^{4,t}_{0\to 0}.$$
Summing this, we get that, conditionally on $U=u,$ and $Z_i(t)=z_i,\; \forall i\le U$ with $\sum z_i=z,$ 
$$Y_{x+}^0\sim \mathcal{Q}^{0,t}_{0\to z}\oplus\mathcal{Q}^{4u,t}_{0\to 0}.$$
On the other hand conditionally on $U=u,$ and $Z_i(t)=z_i,\; \forall i\le U,$ using time reversal and Proposition \ref{additivite_bridge},
$$Y_x^{\delta}\sim \Qc^{\delta,t}_{0\to y-z}=\Qc^{\delta,t}_{0\to 0}\oplus \Qc^{0,t}_{0\to y-z}.$$
Putting everything together, we get that conditionally on $U$, 
$$\Qc^{\delta,t}_{x\to y}=\Qc^{\delta,t}_{0\to 0}\oplus \Qc^{0,t}_{x\to 0} \oplus \Qc^{0,t}_{0\to y} \oplus \Qc^{4u,t}_{0\to 0}.$$
On the other hand, conditionally on $Y$, $Y_0^{\delta}(t)$ has Laplace transform $F^t_{\delta}(\lambda)$ and $Z_i(t)$ are independent random variable with law $\exp\left\{1/(2t)\right\}$
Therefore, using Lemma \ref{cond}, we get the result.

\end{proof}
\begin{Remark}\label{RemarkLaplaceSimplifier}
The expression in Lemma \ref{cond} is fairly general, however in most practical applications $\delta$ is constant by parts. We thus give a simpler expression in this case. Indeed, let $0=\tau_0<\tau_1\dots<\tau_n=t$ be an increasing sequence of times such that on each $[\tau_i,\tau_{i+1}),$ $\delta$ is constant equal to $\delta_i$, then
$$F^t_\delta=\exp\left\{-\frac{\lambda x}{1+2\lambda t}\right\}\prod_{i=0}^{n-1} \left(\frac{1+2\lambda (t-\tau_{i+1})}{1+2\lambda (t-\tau_{i})}\right)^{\delta_i/2}.$$
\end{Remark}

\subsection{About the coefficients in the Bessel bridge decomposition}
The main difference between the classical setting and ours is the fact that the coefficients $b_{\delta,x,y}^t(n)$ are not as explicit. However they can be computed for given values of $\delta$. For numerical purposes it is interesting to have a bound on the speed of decay of these terms. We have the following result:

\begin{Proposition}\label{Eqprop_domination}
Let $\delta\in D$, $d=\inf_{[0,t]}\delta$ and $x,y \in\R^+,$ then for every increasing function $f:\N \mapsto\R^+,$ and with $\nu=d/2-1$ and $z=\sqrt{xy},$
\begin{equation}\label{prop_domination}
\sum_{k=0}^{\infty}b_{\delta,x,y}^t(k) f(k)\le \sum_{k=0}^{\infty}\beta_{\nu,z}(k) f(k),
\end{equation}
where 
$$\beta_{\nu,z}(k)=\left(\frac{z}{2}\right)^{2n+\nu}\frac{1}{\Gamma(n+\nu+1)I_\nu(z) n!}$$
are the coefficients of the $Bessel(\nu,z)$ distribution ($I_\nu(z)$ stands for the modified Bessel function).
In particular the power series associated to the $b_{\delta,x,y}^t(k)$ has an infinite radius of convergence. 
\end{Proposition}
\begin{Remark}
The philosophy of this proposition is that, when one wants to compute quantities related to the Bessel bridge, the truncation arguments used in the classical setting are also valid in our setting.
\end{Remark}
\begin{proof}
From the representation of Lemma \ref{cond}, it is clear that $b_{\delta,x,y}^t(k)$ is the distribution of a variable $U$ which is negatively correlated with $X_*,$ the distribution of the generalized squared Bessel process at time $t$. Proposition \ref{PropoAdditivite} implies that we can couple the generalized squared Bessel process of dimension $\delta$ with a classical squared Bessel process of dimension $d$, the first one being always bigger than the second one. From this we deduce that the distribution of $U$ for $\delta$ is dominated by the distribution for $d$. This distribution is computed in \cite{PY}, (5.j) and is the $Bessel(\nu,z)$ distribution. The result follows easily. In particular the parameter $t$ does not intervene, as it simplifies between $\alpha$ and $\lambda$ in \cite{PY}, (5.j).
\end{proof}

\section{Laplace transform formula associated to extended Bessel bridges.}

We are now interested in computing the Laplace transform of some functionals of the Brownian bridge. We first deal with the simpler case of the bridge to zero. In this case, because additivity holds, things are simpler. We have the following: 
\begin{Proposition}\label{PropLaplPont}
For every Radon measure $\mu$ on $[0,t]$,  there exist $A_0(\alpha)$ and a function $B_0^{\alpha}(u):[0,t]\rightarrow \R$ depending on $\mu$ such that 
\begin{equation*}
\Qc^{\delta,t}_{x\to 0}\left[\exp\left(-\int_0^t \alpha X_ud\mu(u)\right)\right]=(A_0(\alpha))^x\exp\left\{\int_{0}^{t}B_{0}^\alpha(u)\delta_{u}du\right\}.
\end{equation*}

\end{Proposition}
\begin{Remark}
 Note that, for $t=1$, $A_0(\alpha)$ is as in in (5.d') of \cite{PY}. We will give an example of explicit computation of $A_0(\alpha)$ and $B_{0}^\alpha(u)$ in Theorem \ref{Prop-Laplace}.
\end{Remark}
\begin{proof} 
We first assume that the function $\Phi$ defined by (\ref{EqPhiS}) is such that $\int_0^t\left(\frac{\Phi'}{\Phi}\right)^2du<\infty.$ Using Proposition \ref{additivite_bridge}, it is easy to check that $$\phi(x,\delta):=-\log \Qc^{\delta,t}_{x\to 0}\left[\exp\left(-\int_0^t \alpha X_ud\mu(u)\right)\right]$$
 is additive in both terms, therefore $\phi(x,\delta)=\phi(x,0)+\phi(0,\delta),$
and $\phi(x,0)$ being linear, it is necessarily of the form $\phi(x,0)=a_0x.$ We now want to show that 
$$
\phi(0,\delta)=-\log \Qc^{\delta,t}_{0\to 0}\left[\exp\left(-\int_0^t \alpha X_ud\mu(u)\right)\right]
$$
 is of the form "$\int_{0}^{\infty}B_{0}^\alpha(u)\delta_{u}du$". To show this we need to apply the Fr\'echet-Riesz representation, which states (in our context) that any continuous linear form on the Hilbert space $L^2([0,t])$ is of the above mentioned form. So, keeping in mind that $D$ is a dense subset of $L^2([0,t])$, the only thing we need is to show that $\phi(0,\delta)$ is continuous with respect to the $L^2$ norm on $[0,t]$. As we now it is linear, it suffices to show that it is bounded on the unit ball. By proposition \ref{PropoCarmona2}, and the Cauchy-Schwartz inequality, we now that it is true for the unconditioned process, as by assumption $\frac{\Phi'}{\Phi}$ is square integrable. Therefore the result holds as a consequence of the following coupling result, whose interest goes further than the context of this proof.
\begin{Lemma}\label{Domination}
For every $\delta\in D,$ $x\in \R^+,$
there exist a $(\R^+\times \R^+)$ valued process $(X_s,X^0_s)_{s\in [0,t]}$ such that
$$X_s\sim \Qc^\delta_x,\; X_s^0\sim \Qc^{\delta,t}_{x\to 0}, \mbox{ and } X_s^0\le X_s\, a.s.,\, \forall 0\le s\le t.$$
\end{Lemma}
Before giving the proof of this lemma we finish the proof of Proposition \ref{PropLaplPont}. We only need to get rid of the assumption $\int_0^t\left(\frac{\Phi'}{\Phi}\right)^2du<\infty.$ First note that by definition of $\Phi,$ for all $\epsilon$, $\int_\epsilon^{t-\epsilon}\left(\frac{\Phi'}{\Phi}\right)^2du<\infty.$ We deduce from the same arguments as above that for all $\epsilon,$ there exist $B_{0}^{\alpha,\epsilon}$ such that 
$$\Qc^{\delta,t}_{x\to 0}\left[\exp\left(-\int_\epsilon^{t-\epsilon} \alpha X_ud\mu(u)\right)\right]=(A_0(\alpha))^x\exp\left\{\int_{\epsilon}^{t-\epsilon}B_{0}^{\alpha,\epsilon}(u)\delta_{u}du\right\}.$$ 
From this it is clear that for $\epsilon'<\epsilon,$ the restriction to $[\epsilon,t-\epsilon]$ of $B_{0}^{\alpha,\epsilon'}$ is $B_{0}^{\alpha,\epsilon}.$ We can accordingly define $B_{0}^{\alpha}$ on $(0,t)$ such that the restriction to $[\epsilon, t-\epsilon]$ of $B_{0}^{\alpha}$ is $B_{0}^{\alpha,\epsilon}.$ One can then easily check that $B_{0}^{\alpha}$ verifies the statement of Proposition \ref{PropLaplPont}.\end{proof}

We now turn to the Proof of Lemma \ref{Domination}.
\begin{proof}[Proof of Lemma \ref{Domination}]
The proof follows a standard scheme, see for example pages 205-206 of \cite{Faraud}. The idea is to express the conditioned process as a solution of a SDE whose terms are smaller than the SDE defining the unconditioned process, and then apply the Comparison Theorem of \cite{WataComp} to prove the existence of a coupling. 

We have first that for $s\le t,$ denoting by $q_t(x,y)$ the transition kernel associated to $\Qc^\delta_x$. Using \ref{PropoAdditivite}, we can write \begin{equation}q_t(x,y)=\left(q_t^0(x,.)\star q_t^\delta(0,.)\right)(y),\end{equation} where $\star$ stands for the convolution, and $q_t^0$ is the transition kernel of the $0$-dimensional squared Bessel Process. In particular \begin{equation}\label{additivite}q_t(x,0)=q_t^0(x,0)q_t^\delta(0,0).\end{equation} We have
$$\left.\frac{d\Qc^{\delta,t}_{x\to 0}}{d\Qc^\delta_x}\right|_{\mathcal{F}_s}=\frac{q_{t-s}(X_s,0)}{q_t(x,0)}=\exp\left(\log q_{t-s}(X_s,0)-\log q_t(x,0)\right).$$
We write $h(x,s)=\log q_{t-s}(x,0).$ We now from \eqref{additivite} that 
\begin{equation}\label{additivite2}h(x,s)=h^0(x,s)+h^\delta(0,s),\end{equation} where $h^\delta(x,s)$ corresponds to the squared Bessel process of dimension $\delta$.

From this form it is obvious that the derivative of $h$ with respect to $x$ does not depend on $\delta,$ therefore is equal to its value for the $0-$dimensional squared Bessel process, that is $\frac{1}{2(s-t)}$ (see \cite{Faraud} for details), and in particular the second derivative with respect to $x$ is $0$ .  
We have thus, applying Itô's formula that
$$\left.\frac{d\Qc^{\delta,t}_{x\to 0}}{d\Qc^\delta_x}\right|_{\mathcal{F}_s}=\exp\left(-\int_0^s \frac{1}{2(t-u)}dX_u+ \int_0^s \frac {dh}{ds} (X_u,u)du\right).$$
Using once more \eqref{additivite2} we see that $\frac {dh}{ds}$ can be decomposed into a part which only depends on $\delta$ and another one which only depends on $x$. Putting everything together we get
$$\left.\frac{d\Qc^{\delta,t}_{x\to 0}}{d\Qc^\delta_x}\right|_{\mathcal{F}_s}=\exp\left(-\int_0^s \frac{1}{2(t-u)}dX_u+ \int_0^s \frac{X_u}{2(t-u)^2}du + \int_0^s \frac {dh^{\delta}}{ds} (0,u)du)\right).$$
We do not need to compute the last term in the exponential, indeed it is a deterministic term, and we now that the Radon derivative is a martingale, therefore its value is imposed by this condition. Recalling that by definition
 $$dX_u=\delta_u du+2\sqrt{X_u}dW_u,$$ and replacing in the above expression, we have

 $$\left.\frac{d\Qc^{\delta,t}_{x\to 0}}{d\Qc^\delta_x}\right|_{\mathcal{F}_s}=\exp\left(-\int_0^s \frac{\sqrt{X_u}}{(t-u)}dW_u,+ \int_0^s \frac{X_u}{2(t-u)^2}du + \text{ a deterministic function of s}\right).$$
As $X_s-\int_0^s \delta_u du$ is a martingale (under $\Qc^\delta_x$), and keeping in mind that the deterministic term is imposed by the martingale condition, we have:
$$\left.\frac{d\Qc^{\delta,t}_{x\to 0}}{d\Qc^\delta_x}\right|_{\mathcal{F}_s}=\mathcal{E}\left(-\int_0^s \frac{\sqrt{X_u}}{(t-u)}dW_u\right).$$
Girsanov's Theorem then implies that, under $\Qc^{\delta,t}_{x\to 0},$
$$dX_u=\delta_u du+2\sqrt{X_u}dW_u-\frac{2X_u}{t-u} du.$$
Therefore we obtain the result by an application of the Comparison Theorem of \cite{WataComp}.
\end{proof}
We deduce from Proposition \ref{PropLaplPont} and Theorem \ref{ThmBridge} the following
\begin{Corollary}\label{ThmLaplPont}
Let for $\mu$ a Radon measure on $[0,t],$ $I=\int X d\mu$, $A_0(\alpha)$, $B_0^{\alpha}(u):\R^+\rightarrow \R$  as in Proposition \ref{PropLaplPont}, and $\hat{A}_0(\alpha)$ be the $A_0(\alpha)$ for the image of $\mu$ under the map $s\to (t-s)$, let also ${\bf B}_0(\alpha)=\exp\left\{\int_{0}^{t}B_{0}^\alpha(u) du\right\}$ then we have

$$\mathcal{Q}_{x\to y}^{\delta,t}(e^{-\alpha I})=A_0(\alpha)^x \hat{A}_0(\alpha)^y\exp\left\{\int_{0}^{t}B_{0}^\alpha(u)\delta_{u}du\right\}\left(\sum_{n=0}^{\infty} b_{\delta,x,y}^t(n) {\bf B}_0(\alpha)^{4n}\right).$$
\end{Corollary}
We now give explicit computations when $\mu$ is the Lebesgues measure on ${[0,t]}.$
\begin{Theorem}\label{Prop-Laplace}Let $X$ be a $GBESQ^\delta_x$ and let $\beta=\sqrt{2\alpha}t$, then we have 
\begin{multline*}\mathcal{Q}_{x\to y}^{\delta,t}(e^{-\alpha \int_0^t X_s ds})=\exp\left\{\frac{(x+y)}{2t}(1-\beta\coth\beta)\right\}\\ \exp\left\{\int_0^t\left(\frac{1}{2(t-u)}-\frac{\sqrt{2\alpha}}{2}\coth(\sqrt{2\alpha}(t-u))\right)\delta_u du\right\}
\left(\sum_{n=0}^{\infty} b_{\delta,x,y}^t(n) \left(\frac{\beta}{\sinh{\beta}}\right)^{2n}\right).\end{multline*}
\end{Theorem}
\begin{proof}
Here $\mu$ is invariant under the map $s\to (t-s),$ so that $\hat{A}_0(\alpha)=A_0(\alpha).$ Corollary 3.3 on page 465 of \cite{RY} gives for $t=1$, $b \in \R^+$ and constant $\delta$ that
$$
\Qc^{\delta,1}_{x\to 0}\left[\exp\left(-\frac{b^2}{2}\int_0^1X_sds\right)\right]=\left(\frac{b}{\sinh b}\right)^{\frac{\delta}{2}}\exp\left(\frac{x}{2}(1-b \coth b)\right).
$$
Taking first $\alpha:=\frac{b^2}{2}$, then although this expression is given for $t=1$ we can extend it to $[0,t]$ using the fact that, for $\delta$ constant,
$$\mathcal{Q}_{x\to 0}^{\delta,t}\left(e^{-\alpha \int_0^t X_s ds}\right)=\mathcal{Q}_{x\to 0}^{\delta,t}\left(e^{-\alpha t\int_0^{1} X_{st} ds}\right)=\mathcal{Q}_{x/t\to 0}^{\delta,1}\left(e^{-\alpha t^2\int_0^{1} X_{s} ds}\right),$$
where first, we made a change of variable and then used the scaling of Bessel processes. Hence taking $y=0$ and $\delta =0$, we obtain the expression of $A_0(\alpha)$
\begin{equation}\label{EqALaplace}
A_0(\alpha)=\exp\frac{1}{2t}\left\{1-\sqrt{2\alpha}t \coth{\sqrt{2\alpha}t}\right\}.
\end{equation}
And taking $y=0$ and $x=0$ we obtain the expression of  ${\bf B}_0(\alpha)$
\begin{equation}\label{EqBLaplace}
{\bf B}_0(\alpha)=\sqrt{\frac{\sqrt{2\alpha}t}{\sinh{\sqrt{2\alpha}t}}}.
\end{equation}
In order to compute $B_{0}^\alpha(u),$ we take as before test functions for $\delta,$ namely we assume $\delta=\chi_{[s,t]},$ with $0<s<t$.
We have by the same arguments as before
$$\mathcal{Q}_{0\to 0}^{1,1}\left(e^{-\alpha \int_0^t X_s ds}\right)=\mathcal{Q}_{0\to 0}^{1,(t-s)}\left(e^{-\alpha \int_0^{t-s} X_u du}\right)=\mathcal{Q}_{0\to 0}^{1,(t-s)}\left(e^{-\alpha (t-s)\int_0^{1} X_{v(t-s) dv}}\right).$$
By scaling, this is equal to ${\bf B}_0^{1}(\alpha(t-s)^2),$ where ${\bf B}_0^{1}(\alpha)$ is the value of ${\bf B}_0(\alpha)$ for $t=1$.
On the other hand it is equal to $\exp{(\int_s^tB_0^\alpha(u)du)} .$ Then straightforward computations give 
$$B_0^\alpha(u)=\frac{1}{2(t-u)}-\frac{\sqrt{2\alpha}}{2}\coth(\sqrt{2\alpha}(t-u)).$$
This, together with the previous computations, gives the result. Note that $B_0^\alpha(u)$ is not defined for $u=t$, however it is clearly integrable on $[0,t]$.\end{proof}
\begin{Remark}
The previous computations may be easily applied to get formulas for general additive functionals of the GEBSQ. We skip for the moment these tedious computations, but give the complete result in Appendix \ref{LTAF}
\end{Remark}
\section{Some applications}
%%%%%%%%%%%%%%%
%%%%%%%%%%%%%%%
As we said in the introduction, our work was motivated by applications to Finance. Indeed, our class of $GBESQ$ contains various financial models. Hence we can give an answer to many financial mathematical problems using our Bessel bridge decomposition and Laplace transform formula results applied to particular $GBESQ$. We begin thus by giving some financial models which are in the class of $GBESQ$.
\begin{description}
\item[Ornstein-Uhlenbeck (OU):] Let $(V_t)_{t\in [0,T]}$ following an Ornstein-Uhlenbeck model given by
\begin{equation}\label{EqOU}
dV_t=\mu V_t dt + \sigma dW_t \quad \textrm{with } \quad V_0=x.
\end{equation}
Then we have the classical representation in term of $GBESQ$ given by
\begin{equation}\label{RelationOU}
V_t=e^{\mu t}\left[X\left(\frac{-\sigma^2}{2\mu}\left(e^{-2\mu t}-1\right)\right)\right]^{\frac{1}{2}}
\end{equation}
and $X$ is a $GBESQ^{\delta,\tilde{x}}$ with dimension $\delta=1$ and initial value $\tilde{x}=x^2$.

\item[Cox Ingersoll Ross (CIR):] Let $(V_t)_{t\in [0,T]}$ be a CIR process then it is the solution of the stochastic differential equation
$$
dV_t=(\alpha-\beta V_t)dt+\sigma\sqrt{V_t}dW_t \quad \textrm{with} \quad V_0=x.
$$ 
where $\alpha\in \R^+$, $\beta\in \R$, $\sigma>0$. Then we have by a time space transformation the classical representation in term of $GBESQ$ given by
\begin{equation}\label{RelationCIR}
V_t=e^{-\beta t}X\left(\frac{\sigma^2}{4\beta}\left(e^{\beta t}-1\right)\right)
\end{equation}
and $X$ is a $GBESQ^{\delta,x}$ such that 
$$
dX_t=\delta dt +2\sqrt{X_t}d\tilde{W}_t
$$
with a constant dimension $\delta=\frac{4\alpha}{\sigma^2}$ and $\tilde{W}\left(\frac{\sigma^2}{4\beta}\left(e^{\beta t}-1\right)\right)=\int_0^t\frac{\sigma}{2}e^{\frac{\beta s}{2}}dW_s$.
\item[Constant elasticity of variance model (CEV):] Let $(V_t)_{t\in [0,T]}$ following a Constant elasticity of variance model (CEV) given by
\begin{equation}\label{EqCEV}
dV_t=\mu V_t dt + \sigma V_t^\rho dW_t \quad \textrm{with } \quad 0 \leq \rho <1 \quad \textrm{and} \quad V_0=x.
\end{equation}
where $\mu$ represents the instantaneous mean, $\sigma V_t^\rho$ the instantaneous variance of V and $\rho$ is called the elasticity factor. In the limiting case, $\rho=1$, the CEV model returns to the classical Black and Scholes model. In the case where $\rho=0$, it is the Ornstein-Uhlenbeck model and if $\rho=\frac{1}{2}$, it is a Cox-Ingersoll-Ross (CIR) model without mean reverting. Then we obtain that
\begin{equation}\label{RelationCEV}
V_t=e^{\mu t}\left[X\left(\frac{(\rho-1)\sigma^2}{2\mu}\left(e^{2(\rho-1)\mu t}-1\right)\right)\right]^{\frac{1}{2(1-\rho)}}
\end{equation}
and $X$ is a $GBESQ^{\delta,\tilde{x}}$ with dimension $\delta=\frac{2\rho-1}{\rho-1}$ and initial value $\tilde{x}=x^{-2(\rho-1)}$.

\item[Extended CIR:] Let $(V_t)_{t\in [0,T]}$ following an extended Cox Ingersoll Ross model which is given by the solution of the stochastic differential equation (SDE)
\begin{equation}\label{ExCIR}
dV_t=(\alpha_t-\beta_t V_t)dt+\sigma_t\sqrt{V_t}dW_t \quad \textrm{with} \quad V_0=x.
\end{equation}
where $\alpha_t,\beta_t\geq 0$ and $\sigma_t>0$ are time-varying functions and $\sigma_t $ is continuously differentiable with respect to $t\in [0,T]$.
Now, we recall Corollary 3.1 of \cite{SHI02}.
\begin{Corollary}
We have
\begin{equation}\label{RelationCIRextended}
V_t=\exp\left(-\int_0^t\beta_udu\right)X\left(\frac{1}{4}\int_0^t\frac{\sigma_u^2}{\theta_u}du\right),
\end{equation}
where $X$ is a $GBESQ^{\delta,\tilde{x}}$ with time varying dimension $\delta_t=\frac{4(\alpha\circ\nu^{-1})_t}{(\sigma^2\circ\nu^{-1})_t}$ with $\nu_t:=\frac{1}{4}\int_0^t\frac{\sigma_u^2}{\theta_u}du$, $\theta_t=\exp\left(-\int_0^t\beta_udu\right)$ and initial value $\tilde{x}:=\frac{x}{\theta_0}=x$.
\end{Corollary}
\end{description}
Hence, the class of GBESQ contains a large class of standard (financial) models

\begin{Remark}\label{RemarkGBESQDrifte}
We see from these example that one needs in many occasions to evaluate the law of functionals of the form $\int_0^t g(s)X_{f(s)}ds$ for various functions $g$ and $f$ (with $f$ increasing), for $X$ being a GEBSQ or a Generalized Bessel bridge. Laplace transform of these quantities can be computed using the results of Remark \ref{LaplGEBSQ} or alternatively Appendix \ref{LTAF}, using the simple argument that, by change of variable,
$$\int_0^t g(s)X_{f(s)}ds=\int_{f(0)}^{f(t)}\frac{g(f^{-1}(s)))}{f^{'}(f^{-1}(s))}X_s ds.$$

\end{Remark}

\subsection{Generating values of $X_t$ given $X_0$ and its Laplace transform.}\label{SecGenerating}

Let $X$ be a $GBESQ^{\delta,x}$. Our aim is to obtain a sample of generating values of $X_t$ for $t\in [0,T]$ assuming that with know $X_0=x$ and that we have an explicit formula of the Laplace transform of $X_t$. A classic way to do this is to inverse the Laplace transform formula. But in our case the Laplace transform formula we obtained in Corollary \ref{Laplace} is not easy to inverse. So we need an other method to generate a sample of values of $X_t$.  One can use for example the method developed by Ridout in \cite{Ri} which allows to generate sample of random numbers from a distribution specified by its Laplace transform (without having to inverse the Laplace transform). Hence, as we can evaluate explicitly the Laplace transform of X assuming $X_0=x$ with our Corollary \ref{Laplace} and simplification given by Remark \ref{RemarkLaplaceSimplifier}, one can apply the method of Ridout in \cite{Ri} to obtain a sample of values of $X_t$ for any $t\in [0,T]$.

This method also permits to jointly simulate $X_t$ (for a fixed $t$) and any additive functional of $X_s, \,0\le s\le t$ without simulating the whole process. To do this first simulate $X_t$ through its Laplace transform, then plug this value into the formula of Appendix \ref{LTAF} and repeat the process. We give an example in Section \ref{Simulationcouple}

% After, using Monte Carlo method, we can apply our Theorem \ref{Prop-Laplace} for this sample and obtain an explicit evaluation of the Laplace transform of a $GBESQ^{\delta,x}$ with respect to its endpoints: $X_0=x$ and $X_t=y$.
\subsection{Simulation of integrate process conditioning to its starting and ending points.}\label{Simulationcouple}
In this application, we follow results found by Broadie and Kaya in \cite{Broadie} and extended by Glasserman and Kim in \cite{Gla}. We will apply the same idea but using our more general framework. Indeed, assume that we would like to simulate a stochastic volatility model given for all $t\in [0,T]$ , $\delta \in D$ by
\begin{eqnarray}\label{EqS}
dS_t&=&\mu_tS_tdt+\sqrt{V_t}\left(\rho dW^1_t+\sqrt{1-\rho^2}dW^2_t\right),\quad S_0=s,\\
dV_t&=&\delta_tdt+2\sqrt{V_t}dW^1_t, \quad V_0=x.
\end{eqnarray}
where $W^1$ and $W^2$ are two independent standard Brownian motion, $X_t$ is a Markov jump process on finite space $\Sc:=\{1,2,\dots,K\}$ and $\rho \in [-1,1]$. So the stochastic volatility process is a $GBESQ^{\delta,x}$. Let for all $t\in [0,T]$ the solution of the first stochastic differential equation in \eqref{EqS}:
$$
S_t=S_0\exp\left(\int_0^t \mu_sds -\frac{1}{2}\int_0^t V_sds +\rho \int_0^t \sqrt{V_s}dW^1_s +\sqrt{1-\rho^2}\int_0^t \sqrt{V_s}dW_s^2\right).
$$
Moreover $\int_0^t \sqrt{V_s}dW^1_s=\frac{1}{2}\left(V_t-V_0-\int_0^t\delta_sds\right)$. Then we have that $\log\left(\frac{S_t}{S_0}\right)$ is conditionally normal given $\int_0^t V_sds$, $V_t$ and $V_0$. Hence
$$
\log\left(\frac{S_t}{S_0}\right)\sim \Nc\left(\int_0^t \mu_sds -\frac{1}{2}\int_0^t V_sds-\frac{\rho}{2}\left(V_t-V_0-\int_0^t\delta_sds\right),(1-\rho^2)^2\int_0^t V_sds\right).
$$
Moreover we can sample $V_t$ using again method stated in subsection \ref{SecGenerating}. Hence we can obtain a sample of couple $(V_0,V_t)$. Then the problem of exact simulation of our model $\eqref{EqS}$ is reduced to the problem of sampling from 
$$
\left(\int_0^t V_sds\vert V_0,V_t\right).
$$

Since $V$ is a $GBESQ^\delta_x$, we know by Theorem \ref{ThmBridge} that for every $\delta\in\D,$ $x\in \R^+$, $y\in \R^+$, there exist processes $X_1$, $X_2$, $X_3$ and $X_4$ such that
\begin{equation}\label{EqdecoGla}
\left(\int_0^t V_sds\vert V_0=x,V_t=y\right)=X_1+X_2+X_3+X_4,
\end{equation}
where
$X_1 \sim \mathcal{Q}^{0,t}_{x\to 0}$, $X_2\sim \mathcal{Q}^{0,t}_{0\to y}$, $X_3\sim \mathcal{Q}^{\delta,t}_{0\to 0}$ and $X_4\sim \sum_{n=0}^\infty b_{\delta,x,y}^t(n)\mathcal{Q}_{0\to 0}^{4n,t}$. Moreover, depending on the $GBESQ$ $V$ chosen, we can use Theorem \ref{Prop-Laplace} or Remark \ref{RemarkGBESQDrifte}  and Proposition \ref{Eqprop_domination} to have a way of sampling from $\left(\int_0^t V_sds\vert V_0,V_t\right)$; and so an exact simulation of our stochastic model \eqref{EqS}.

\subsection{Evaluation of Bond price in a credit notation model with regime switching}\label{EvalBond}
Assume that we are under an equivalent risk neutral probability $\tilde{P}$, then from the general asset pricing theory, we have that the discount bond price $B(T)$ at time $0$ with maturity $T>0$ of an interest rate $r$ is given by
\begin{equation}\label{Bond}
B(T)=\tilde{\E}\left[\exp\left(-\int_0^T r_sds\right)\vert r_0=r\right]
\end{equation}

Let the process $r$ be given by the solution of the stochastic differential equation given by \eqref{ExCIR}, we assume that the value of the parameters $(\alpha_t,\sigma_t)$ is function for all $t\in [0,T]$ of the value of an exogenous continuous time on finite state Markov chain $(S)_{t\in [0,T]}$. Moreover, we assume that $\beta_t\equiv 0$. Hence the Markov chain $S$ takes value in $\Sc:=\{1,2,\dots ,N\}$ and could represent the credit notation of a firm A given by a credit notation entity. Then the process $r$ could be the spread of the firm A and we would like to evaluate the price at time $t\in [0,T]$ of a bond of this firm with maturity $T>0$. (See Goutte and Ngoupeyou in \cite{GN} for more details about this defaultable regime switching modeling).
Hence our model becomes
\begin{equation}\label{CIRRS}
dr_t=\alpha(S_t)dt+\sigma(S_t)\sqrt{r_t}dW_t,\quad r_0=r
\end{equation}
and so $\alpha,\beta$ and $\sigma$ are piecewise continuous functions.

\begin{Remark}
\begin{enumerate}
\item The fact that we assume that $S$ is a exogenous Markov chain implies that $S$ is independent of the Brownian motion $W$ for all $t\in [0,T]$.
\item Assuming that $S$ is exogenous has a economic sense, indeed in representation \eqref{CIRRS} there is two source of randomness, firstly the market risk modeled by the Brownian motion $W$ and secondly an external expertise risk factor $S$ given by an external entity of the firm A. This could be the credit notation representing the wealth at time $t\in [0,T]$ of the firm A or an indicator of the wealth of the global international economy at time t.
\end{enumerate}
\end{Remark}
Let us denote by $\bold S$ the historical values of S, i.e. $\bold S:=(S_t; t\in [0,T])$. Using \eqref{RelationCIRextended}, we have that 
$$
r_t=X\left(\frac{1}{4}\int_0^t\sigma_s^2ds\right)
$$
where $X$ is a $GBESQ^{\delta,\tilde{x}}$ with time varying dimension $\delta_t=\frac{4(\alpha\circ\nu^{-1})_t}{(\sigma^2\circ\nu^{-1})_t}$ with $\nu_t:=\frac{1}{4}\int_0^t\sigma_s^2ds$ and initial value $\tilde{x}=r$. Hence we have two different ways to evaluate
$$
B^A(T)=\tilde{\E}\left[\exp\left(-\int_0^T r_sds\right)\vert r_0=r;\bold S\right]
$$
Firstly, we can use the Proposition \ref{PropoCarmona2} and solve the corresponding equation \eqref{EqPhiS}. Or we can use our Bessel Bridge decomposition to jointly sample the final value of the process $X$ and the Laplace transform. Indeed, by Remark \ref{RemarkGBESQDrifte}, we have in this model that
\begin{equation}\label{Eqf}
f(t)= \nu_t:=\frac{1}{4}\int_0^t\sigma_s^2ds\quad \textrm{and} \quad r_t=X_{f(t)}
\end{equation}
and $g(.)\equiv 1$. So we can apply Theorem \ref{ThmA6}, which is the extended version of our Theorem \ref{Prop-Laplace}, to evaluate the Laplace transform of $\int_0^t X_{f(s)}ds$ using our extended Bessel bridge decomposition where we take for Radon measure $\mu$ the measure given for all $t\geq 0$ by:
$$
d\mu(t)=\frac{1}{f^{'}(f^{-1}(t))}dt.
$$

Hence to apply this result, we need to have a sample of value of $r_T=X_{f(T)}$ with respect to the starting point $r_0=X_0:=x$ and the law of the associated $GBESQ$ for $r$. For this we will use the method stated in subsection \ref{SecGenerating} since we have a Laplace transform formula expression given in Corollary \ref{Laplace}. Moreover the quantity $b_{\delta,x,y}^t(n)$ given by Lemma \ref{cond} is in this case simpler since we have Remark \ref{RemarkLaplaceSimplifier}.
 Finally these allow us to obtain an evaluation of the Bond price $B^A(T)$ at time zero of the firm A using Monte Carlo.

\subsection{Simulation of the default time in credit risk model}

We work in the model of construction of Cox processes with a given stochastic intensity $\lambda$. Hence we denote by $\lambda$ a stochastic default intensity of a firm or country A. We assume for all $t\in [0,T]$ that $\int_0^t \lambda_sds<\infty$ and that $\int_0^\infty \lambda_sds=\infty$. Then the default time $\tau$ of the firm or country A is a real random variable which is given by the first time when the increasing process $\int_0^t \lambda_sds$ is below a uniform random variable $U$ independent with $\lambda$. Hence we get
\begin{equation}\label{EqTau}
\tau:=\inf\left\{t\geq 0 ; \exp\left(-\int_0^t \lambda_sds\right)\leq U\right\}.
\end{equation}
So to simulate the default time requires to simulate the stochastic default intensity $\lambda$ together with its time integral. Hence assuming that $\lambda$ is a $GBESQ^{\delta,x}$, we can use again our Bessel bridge decomposition given in Theorem \ref{ThmBridge}, our Laplace transform formula given in Corollary \ref{Laplace} and the method stated in subsection \ref{SecGenerating}, to do this.

\appendix
\section{Laplace transform of additive functionnals of the GEBSQ}
\label{LTAF} In this appendix we give an explicit formula for the Laplace transform of additive functionals of the Generalized Bessel bridge. 
We first recall Corollary \ref{ThmLaplPont}:
\begin{Cita}
Let for $\mu$ a Radon measure on $[0,t],$ $I_\mu=\int X d\mu$, $A_0(\alpha)$, $B_0^{\alpha}(u):\R^+\rightarrow \R$  as in Proposition \ref{PropLaplPont}, and $\hat{A}_0(\alpha)$ be the $A_0(\alpha)$ for the image of $\mu$ under the map $s\to (t-s)$, let also ${\bf B}_0(\alpha)=\exp\left\{\int_{0}^{t}B_{0}^\alpha(u) du\right\}$ then we have

$$\mathcal{Q}_{x\to y}^{\delta,t}(e^{-\alpha I_\mu})=A_0(\alpha)^x \hat{A}_0(\alpha)^y\exp\left\{\int_{0}^{t}B_{0}^\alpha(u)\delta_{u}du\right\}\left(\sum_{n=0}^{\infty} b_{\delta,x,y}^t(n) {\bf B}_0(\alpha)^{4n}\right).$$
\end{Cita}

The idea, used before, is to take test functions for $\delta$ to get the explicit value of $A_0(\alpha)$ and $B_{0}^\alpha(u)$.
We recall the result in the case of Standard Bessel Process, as given for example on page 465 of \cite{RY}, in the proof of Theorem 3.2., taking $y=0$ and recalling that $\lim_{x\to 0}\frac{I_{\nu}(xy)}{I_{\nu}(x)}=y^{\nu},$
\begin{Theorem}\label{Theoremdimcst}In the case of constant dimension $\delta$, we have
$$\mathcal{Q}_{x\to 0}^{\delta,1}[e^{- I_\mu}]=\exp\left\{\frac{1}{2} \left(F_\mu (0) +1-\frac{1}{\int_0^1\phi_\mu (s)^{-2}ds}\right)x\right\}\left(\frac{1}{\int_0^1\phi_\mu (s)^{-2}ds}\right)^{\delta/2},$$
where, as usual, $\phi_\mu$ is the unique solution of $\phi''=\mu\phi$ which is positive, non increasing and such that $\phi_\mu(0)=1;$ and $F_\mu=\frac{\phi'}{\phi}.$ 
\end{Theorem}
We first extend this formula to bridges of arbitrary length (keeping for the moment $\delta$ constant).
Let $\mu$ be a positive Radon measure on $[0,t]$. We have by change of variable, then scaling,
\begin{equation*}
\mathcal{Q}_{x\to y}^{\delta,t}(e^{- I_\mu})= \mathcal{Q}_{x\to 0}^{\delta,t}(e^{-\int_0^t X_s d\mu(s)})=\mathcal{Q}_{x\to 0}^{\delta,t}(e^{-\int_0^1 X_{ts} t d\mu(ts)})=\mathcal{Q}_{x/t\to 0}^{\delta,1}(e^{-\int_0^1 X_{s} t^2 d\mu(ts)}).
\end{equation*}
By Theorem \ref{Theoremdimcst}, we deduce that 
$$\mathcal{Q}_{x\to 0}^{\delta,t}[e^{- I_\mu}]=\exp\left\{\frac{1}{2} \left(F_\mu (0) +1-\frac{1}{\int_0^1\phi_\mu (s)^{-2}ds}\right)\frac{x}{t}\right\}\left(\frac{1}{\int_0^1\phi_\mu (s)^{-2}ds}\right)^{\delta/2},$$
where $\phi_\mu$ is the unique solution of $\phi''=t^2\mu(t\cdot )\phi$. It is clear that $\phi_\mu=\phi_\mu^t(t\cdot),$ where
$\phi_\mu^t$ is the unique solution of $\phi''=\mu\phi$ which is positive, non increasing and such that $\phi_\mu^t(0)=1.$
Expressing the previous result in term of $\phi_\mu^t$ we have, with obvious notations,
$$\mathcal{Q}_{x\to 0}^{\delta,t}[e^{- I_\mu}]=\exp\left\{\frac{1}{2} \left(F_\mu^t (0) +\frac{1}{t}-\frac{1}{\int_0^t\phi_\mu^t (s)^{-2}ds}\right)x\right\}\left(\frac{t}{\int_0^t\phi_\mu^t (s)^{-2}ds}\right)^{\delta/2}.$$
We deduce from this expression that $A_0(1)=\exp\left\{\frac{1}{2} \left(F_\mu^t (0) +\frac{1}{t}-\frac{1}{\int_0^t\phi_\mu^t (s)^{-2}ds}\right)\right\}.$
We also deduce that $\hat{A}_0(1)=\exp\left\{\frac{1}{2} \left(\hat{F}_\mu^t (0) +\frac{1}{t}-\frac{1}{\int_0^t\hat{\phi}_\mu^t (s)^{-2}ds}\right)\right\},$
where $\hat{\phi_{\mu}^t}$ is the unique solution of $\phi''=\mu(t-\cdot)\phi$ which is positive, non increasing and such that $\hat{\phi_{\mu}^t}(0)=1;$ and $\hat{F}^t_\mu=\frac{\hat{\phi_{\mu}^t} '}{\hat{\phi_{\mu}^t}}.$

We now turn to the computation of $B_{0}^1(u)$. For this take $\delta=\chi_{[b,t]}.$ We have as previously explained,

$$\exp\left\{\int_{b}^{t}B_{0}^\alpha(u)du\right\}=\mathcal{Q}_{0\to 0}^{1,t-b}\left[e^{- \int_0^{t-b}X_s d\mu(s+b)}\right]=\left(\frac{t-b}{\int_0^{t-b}\phi_\mu^{t,b} (s)^{-2}ds}\right)^{1/2},$$
where $\phi_\mu^{t,b}$ is the unique solution of $\phi''=\mu(b+\cdot)\phi$ which is positive, non increasing and such that $\phi_\mu^{t,b}(0)=1.$ therefore $\phi_\mu^{t,b}=\frac{\phi_\mu^{t}(b+.)}{\phi_\mu^t(b)}.$
Therefore we have
\begin{equation}\label{represb}\exp\left\{\int_{b}^{t}B_{0}^1(u)du\right\}=\mathcal{Q}_{0\to 0}^{1,t-b}\left[e^{- \int_0^{t-b}X_s d\mu(s+b)}\right]=\left(\frac{t-b}{\phi_\mu^{t}(b)^2\int_b^{t}\phi_\mu^{t} (s)^{-2}ds}\right)^{1/2}.\end{equation}
We take the logarithm and differentiate with respect to $b$ to get:
$$B_{0}^1(u)=\frac{1}{2(t-u)}+ F_\mu^t(u)-\frac{\phi_\mu^t(u)^{-2}}{2\int_u^{t}\phi_\mu^{t} (s)^{-2}ds}.$$
Finally, using (\ref{represb}), we have
$${\bf B}_0(1)=\exp\left\{\int_{0}^{t}B_{0}^1(u) du\right\}=\left(\frac{t}{\int_0^t\phi_\mu^t (s)^{-2}ds}\right)^{1/2}.$$
Finally, putting everything together, we have:
\begin{Theorem}\label{ThmA6}
\begin{equation*}
\begin{aligned}
\hspace{-1.5cm}\mathcal{Q}_{x\to y}^{\delta,t}(e^{-I_\mu})=\exp\frac{x}{2} \left(F_\mu^t (0) +\frac{1}{t}-\frac{1}{\int_0^t\phi_\mu^t (s)^{-2}ds} \right)\exp\frac{y}{2} \left(\hat{F}_\mu^t (0) +\frac{1}{t}-\frac{1}{\int_0^t\hat{\phi}_\mu^t (s)^{-2}ds}\right)
\\\exp\left\{\int_{0}^{t}\left(\frac{1}{2(t-u)}+ F_\mu^t(u)-\frac{\phi_\mu^t(u)^{-2}}{2\int_u^{t}\phi_\mu^{t} (s)^{-2}ds}\right)\delta_{u}du\right\}\left(\sum_{n=0}^{\infty} b_{\delta,x,y}^t(n) \left(\frac{t}{\int_0^t\phi_\mu^t (s)^{-2}ds}\right)^{2n}\right),
\end{aligned}
\end{equation*}
where
$\phi_\mu^t$ is the unique solution of $\phi''=\mu\phi$ which is positive, non increasing and such that $\phi_\mu^t(0)=1,$
 $\hat{\phi_{\mu}^t}$ is the unique solution of $\phi''=\mu(t-\cdot)\phi$ which is positive, non increasing and such that $\hat{\phi_{\mu}^t}(0)=1;$ and $F_\mu^t=\frac{{\phi_{\mu}^t}'}{\phi_{\mu}^t},$ $\hat{F}^t_\mu=\frac{\hat{\phi_{\mu}^t} '}{\hat{\phi_{\mu}^t}}.$ 
\end{Theorem}

\end{document}